\documentclass[preprint, 12pt, final]{elsarticle}

\usepackage{epsfig}
\usepackage{subfig}
\usepackage{graphicx}

\usepackage[hidelinks]{hyperref}
\usepackage{url}

\usepackage{amssymb}
\usepackage{amsthm}
\usepackage{amsmath}

\theoremstyle{plain}% Theorem-like structures
\newtheorem{theorem}{Theorem}[section]

\theoremstyle{definition}

\usepackage{booktabs} % fancy table
\usepackage{colortbl}
\usepackage[table]{xcolor}

\usepackage{tcolorbox}

\usepackage{lineno,hyperref}
\usepackage[toc,page,titletoc]{appendix}

\journal{Mathematical Medicine and Biology}

\begin{document}

\begin{frontmatter}

\title{Impact of a cost functional on the optimal control and the cost-effectiveness: control of a spreading infection as a case study}

%\title{Optimal Control and Cost-Effectiveness Analysis for a Simple Epidemic Model: The Role of the Cost Functional}

\author[address1]{Fernando Salda\~na}
%\cortext[correspondingauthor]{Corresponding author}
%\ead{fernando.saldana@im.unam.mx}

\author[address2]{Jos\'e Ariel Camacho-Guti\'errez}
\ead{ariel.camacho@uabc.edu.mx}

\author[address3]{Andrei Korobeinikov}
%\ead{akorobeinikov@crm.cat}

\address[address1]{Instituto de Matem\'aticas, Campus Juriquilla, 76230, Universidad Nacional Aut\'onoma de M\'exico, Qu\'eretaro, Mexico}
\address[address2]{Facultad de Ciencias, Universidad Aut\'onoma de Baja California, 22860 Baja California, Mexico}
\address[address3]{School of Mathematics and Information Science, Shaanxi Normal University, Xi\'an, China}

\begin{abstract}
In applications of the optimal control theory to problems in medicine and biology, the dependency of the objective functional on the control itself is often a matter of controversy. In this paper, we explore the impact of the dependency using reasonably simple \emph{SIR} and \emph{SEIRS} epidemic models. To qualitatively compare the outcomes for different objective functionals, we apply the cost-effectiveness analysis. Our result shows that, at least for the comparatively inexpensive controls, the variation of the power at the controls in a biologically feasible range does not significantly affect the forms of the optimal controls and the corresponding optimal state solutions. Moreover, the costs and effectiveness are affected even less. At the same time, the dependency of the cost on the state variables can be very significant.  
\end{abstract}


\end{frontmatter}

\section{Introduction}
Application of the optimal control theory to control problems that arise in medicine and biology can potentially bring significant benefits, such as reduction of the cost and side effects of an applied intervention policy. However, so far, the amount of publications addressing such problems is rather insufficient and, what is more important, the existing work attracted a lot less attention to the biologists and medical scientists that they deserved. To a certain extent, this can be explained by the lack of synergy between the disciplines and the difficulty of understanding a complicated mathematical technique. Nevertheless, the authors believe that there is another and, probably, a more serious reason for the lack of enthusiasm towards the results of the optimal control theory. The problem is that the results are usually obtained for a very specific objective functional, and the extents to which these particular results can be trusted, if this functional is alerted, are generally unclear. In mathematical terms, it is a problem of robustness and reliability of the results with respect to perturbations of the objective functional. While the mathematical theory of robustness for the optimal control problems is a direction of intensive research, it is far from completion.  Yet, the relevance of the real-life applications urges to explore this problem using simpler approaches, such as direct computations. 

A particular issue that is a very common source of constant controversy in biomedical applications is a form of the dependency of the cost function on the considered controls (for example, see discussion and the bibliography in \citep{sharomi2017optimal,diCosts}). The problem is that, in the majority of biomedical applications, the actual form of the costs and the dependency on these on the controls are usually unclear and can hardly be defined with a necessary degree of accuracy. Moreover, the functional dependency can be different under different circumstances. As a result, for the sake of simplicity, in many publications in the literature, the cost is postulated to be equal to a sum of the weighted squares of the considered controls. (Such functionals are typically referred to as the $L_{2}$-type functionals.) Such a formulation is the most mathematically convenient, as for such $L_2$-type functionals, the optimal controls can be obtained as explicit functions of the state and adjoint variables by virtue of the Pontryagin's maximum principle \citep{saldana2019optimal}. Then, the optimal control problem can be straightforwardly reduced to a two-point boundary value problem, which is suitable for numerical solving. Thus, the main advantage of cost functionals with a quadratic term in the control is its mathematical simplicity. However, in biomedical applications, the interpretation of such $L_2$-type functionals is not always clear and their use is often difficult to justify.  
(In engineering, where this type of objective functional came from, the square of the controls means the energy spent on the control action~\citep{grigorieva2018optimal}.) 

The problem was properly realized and other forms of dependency were also considered. For instance, Ledzewicz et al.~\citep{ledzewicz2004comparison} argued that the $L_1$-type functionals, that is, the cost functionals with a linear dependence on the control, may capture the cost of an intervention policy more accurately than the $L_2$-type functionals. 
Di Liddo~\citep{diCosts} remarked that, if $u(t)$ is the fraction of a population $N(t)$ treated by some drug and $p$ is the unitary cost of the drugs, then a natural way to define the cost $C$ of the treatment is $ C=pu(t)N(t)$. 
Moreover, as an ultimate way to exclude the controversy of the intervention cost, Grigorieva et al.~\citep{grigorieva2020,grigorieva2016optimal,grigorieva2014optimal,grigorieva2013,grigorieva2012optimal} suggested to completely disregard the cost of the intervention under the assumption that it is ''tolerable'' and small compared to the losses that can be inflicted otherwise. However, one has to take into consideration that the optimal control with the $L_1$-type functionals (and in particular with the cost depending on the state variables) and the functionals independent of the controls are generally less amenable to the mathematical analysis, and usually involves singular and bang-bang controls. Moreover, the linear dependency on the control might be also not the most realistic option, as some form of the nonlinearity of the cost is likely to occur, for example, due to the overload of the health service. 

These considerations indicate that further investigation is needed in order to understand how the objective functional affects the solutions of optimal control problems, knowledge that can lead to the construction of practical disease management strategies. Moreover, such a study is essential to verify the validity and practicality of a large corpus of the results obtained for $L_1$- and $L_{2}$- types objective functionals. Another similar related problem is a form of the functional dependency of the cost on the state variables and the impact of this on the optimal control. However, despite the apparent practical relevance of this problem, very little mathematical work was so far done in this direction: so far, the authors were able to find only two publications~\citep{ledzewicz2004comparison,ledzewicz2020comparison} that address this issue.  

In this paper, we are addressing this issue considering the problem of the control of spreading infectious disease as a case study.  We consider this problem for two classical epidemic models, namely a $SIR$ model and a more complex $SEIRS$ model. For these models, we formulate four optimal control problems that correspond to four objective functionals that are most frequent in the literature on disease control, and then analyze the outcomes. We have to point out that quantitative analysis of the results for different objective functionals is a nontrivial problem by itself: in this paper, in order to quantitatively compare the outcomes for different cost functions, we employ the cost-effectiveness analysis. (The cost-effectiveness analysis was successfully applied to the epidemiological problem by a number of authors; for instance, see \citep{agusto2019optimal, biswas2017, grigorieva2020optimal, muennig2016, okosun2013optimal, rodrigues2014cost, silva2013optimal} and bibliography therein.) The main goal of this study is (i) to investigate how a specific form of the objective functional impacts the resulting optimal control, and (ii) to identify if there is a type of objective for which the optimal control solutions have better performance in terms of cost-efficacy to assess the advantages or disadvantages of selecting one functional over the others. 

The rest of the paper is structured as follows. In the next section, we study the solutions of an optimal control problem for a classical $SIR$ model with a single time-dependent control and  in section \ref{sec:SIRsimulations} we conduct a numerical study of the optimal control solutions. In section \ref{sec:CEAsinglecontrol}, we carry out a cost-effectiveness analysis to investigate the cost and health benefits of the control interventions resulting from the solution of the optimal control problem for the $SIR$ model. In section \ref{sec:OCseirs}, we construct a $SEIRS$ model with two time-dependent controls and formulate an optimal control problem. Section \ref{sec:CEAseirs} contains a cost-effectiveness analysis based on the incremental cost-effectiveness ratio (ICER) to measure the performance of the control interventions for the $SEIRS$ model. Finally, a detailed discussion of the results is presented in section \ref{sec:discussion}.

\section{Optimal control for a $SIR$ model with a single control}\label{sec:OCsinglecontrol}
As a case study, let us consider the classical \emph{SIR} epidemic model with vital dynamics given by the following initial value problem:
\begin{equation}
\begin{aligned}
\dfrac{dS}{dt}&=\mu N -\beta \dfrac{SI}{N}-\mu S,\\
\dfrac{dI}{dt}&=\beta \dfrac{SI}{N}-\gamma I -\mu I,\\
\dfrac{dR}{dt}&=\gamma I -\mu R,\label{SIRmodel}
\end{aligned}
\end{equation}
with non-negative initial conditions
\begin{equation}
S(0)=S_{0}\geq 0, \quad I(0)=I_{0}\geq 0, \quad R(0)=R_{0}\geq 0. 
\end{equation}
Here, $S(t)$ is the number of susceptibles at time $t$, $I(t)$ the number of infectious, and $R(t)$ the number of recovered individuals. The total population $N(t)$ is the sum of the sizes of these three classes, namely, $N(t)=S(t)+I(t)+R(t)$.
In system \eqref{SIRmodel}, the parameters $\beta$ and $\gamma$ represent the transmission and the recovery rates, respectively. Individuals are recruited into the population as susceptibles at a rate $\mu N$. The natural death of all individuals, whatever their status, occurs also at a per capita rate $\mu$. Thus, the deaths balance the births, so $N(t)$ is constant and equal to $N_{0}=N(0)$.\par 
Our main interest here is to identify the optimal healthcare intervention for curtailing the spread of the infection at the minimal cost. Therefore, we modify the classical \emph{SIR} model \eqref{SIRmodel} to include a time-dependent control function $u(t)$ that represents the treatment of the infectious, and can be considered as the per capita treatment rate of infected individuals. The \emph{SIR} model with control becomes
\begin{equation}
\begin{aligned}
\dfrac{dS}{dt}&=\mu N -\beta \dfrac{SI}{N}-\mu S,\\
\dfrac{dI}{dt}&=\beta \dfrac{SI}{N}-\gamma I -u(t)I -\mu I,\\
\dfrac{dR}{dt}&=\gamma I + u(t) I -\mu R.\label{ControlledSIR}
\end{aligned}
\end{equation}\par 
Public health authorities wish to minimize both the cost of infection and the cost of implementing the control during a time interval $[0,t_{f}]$. Therefore, we consider the following objective functional:
\begin{equation}
J_{i}(u(t))=\int_{0}^{t_{f}}\varphi(I(t))+\sigma_{i}(u(t),I(t))dt.\label{CostFunctional}
\end{equation}
Here, function $\varphi(I)$ represents the cost of the reduction in health and well-being, that is, the cost of infection. These costs related to pain and suffering (quality of life) are sometimes referred to as morbidity costs \citep{muennig2016}. On the other hand, the function $\sigma_{i}(u,I)$ represents the costs of healthcare services, that is, the costs of vaccines, treatment, hospitalizations and other resources.\par 
For simplicity, we assume that the cost of infection is proportional to the size of the infected class, thus
\begin{equation}
\varphi\left( I(t) \right) = A_{1} I(t).\label{C1}
\end{equation}
The above simplification allows us to focus on the impact of the cost function $\sigma_{i}(u,I)$ in the form of the optimal control solution. Following the ideas introduced in \citep{diCosts}, we consider the following forms for the cost function $\sigma_{i}(u,I)$:
\begin{eqnarray}
\sigma_{1}(u,I)&=&A_{2}u^{2}, \quad \text{quadratic state independent cost (QSI)},\label{c1}\\
\sigma_{2}(u,I)&=&A_{2}u^{2}I, \quad \text{quadratic state dependent cost (QSD)}\label{c2}.\\
\sigma_{3}(u,I)&=&A_{2}u, \quad \text{linear state independent cost (LSI)},\label{c3}\\
\sigma_{4}(u,I)&=&A_{2}uI, \quad \text{linear state dependent cost (LSD)}.\label{c4}
\end{eqnarray}
The coefficients $A_{i}>0$, $i=1,2$ are positive weights on the costs. In real-life situations, the monetary costs and side effects of a healthcare intervention or therapy are commonly small compared with the potential losses that an epidemic can inflict. Therefore, we assume $A_{1}>A_{2}$. In particular, for later numerical simulations, we will take $A_{1}=100$ and $A_{2}=10$. We have to remark that, in the opposite case, the monetary costs of the intervention usually limit in excess the use of a therapy even to the point of making the policy ``not treat any infected people at all'' the most cost-effective strategy \citep{diCosts}.
We have to find an optimal control $u^{*}(t)$ such that 
\begin{equation}
J_{i}(u^{*}(t))=\min_{u\in \mathcal{U}}\left\lbrace J_{i}(u(t))\right\rbrace \;
\text{subject to system}\; \eqref{ControlledSIR}\label{OCproblem}
\end{equation}
where the set of admissible controls is
\begin{equation}
\mathcal{U}=\left\lbrace u:[0,t_{f}]\rightarrow [0,1]\; \text{is Lebesgue measureble}\right\rbrace .\label{ControlSet}
\end{equation}
Let us introduce the vector of states $x=(S,I,R)^{T}$. We call a pair of states and control $(x,u)$ satisfying both \eqref{ControlledSIR} and $u\in \mathcal{U}$ a feasible pair. Theorem $4.1$ in \citep[Chapter III]{Fleming1975} ensures the existence of an optimal control $u^{*}(t)$ and corresponding optimal solution $x^{*}=(S^{*},I^{*},R^{*})^{T}$ for each of our control problems $\lbrace QSI, QSD, LSI, LSD\rbrace$ (the details are given in the Supplementary material \ref{AppendixExistence}).\par 
The Pontryagin's maximum principle states that it is necessary for any optimal control $u^{*}(t)$ along with the optimal state trajectory $x^{*}(t)=(S^{*},I^{*},R^{*})^{T}$ to solve the so-called Hamiltonian system. In other words, this principle converts the optimal control problem \eqref{OCproblem} into a problem of minimizing pointwise a Hamiltonian $H$ with respect to the control $u(t)$. The Hamiltonian is defined in terms of the integrand of the cost functional $J$ and the right-hand side of the controlled \emph{SIR} model \eqref{ControlledSIR} as follows:
\begin{equation}
\begin{aligned}
H_{i} &= A_{1}I + \sigma_{i}(u,I) 
  + \lambda_{S}\left( \mu N -\beta \dfrac{SI}{N}-\mu S\right)\\
  &+ \lambda_{I}\left(\beta \dfrac{SI}{N}-\gamma I -u(t)I -\mu I\right) 
  + \lambda_{R}\left(\gamma I + u(t) I -\mu R\right) ,
\end{aligned}
\end{equation}
where $\lambda_{S}$, $\lambda_{I}$ and $\lambda_{R}$ are the associated adjoint variables.\par 
We apply the Pontryagin's maximum principle to derive the necessary conditions that every optimal control satisfies.

\begin{theorem}\label{theorem:pont}
Let $u^{*}(t)$ and $x^{*}(t)$  be the optimal control and the corresponding optimal solutions for the optimal control problem \eqref{ControlledSIR}, \eqref{OCproblem}. Then there exists a piecewise differentialble adjoint variable $\lambda(t)$ such that
\begin{equation}
H_{i}(t,x^{*},u(t),\lambda(t))\geq H_{i}(t,x^{*}(t),u^{*}(t),\lambda(t))
\end{equation}
holds at any $t\in[0,t_{f}]$ for all controls $u\in \mathcal{U}$.\par 
Moreover, except at the points of discontinuities of $u^{*}(t)$, the adjoint variable $\lambda=(\lambda_{S},\lambda_{I},\lambda_{R})^{T}$ satisfies equations 
\begin{equation}
\dfrac{d\lambda_{j}}{dt}=-\dfrac{\partial H_{i}}{\partial x_j}(t,x^{*},u^{*},\lambda),\quad j=S,I,R.
\end{equation}
Therefore,
\begin{equation}
\begin{aligned}
\dfrac{d\lambda_{S}}{dt}&=\lambda_{S}\left( \beta\dfrac{I^{*}}{N}+\mu\right) -\lambda_{I}\beta\dfrac{I^{*}}{N},\\
\dfrac{d\lambda_{I}}{dt}&=-A_{1}- \dfrac{\partial \sigma_{i}}{\partial I}(u^{*},I^{*})+ \lambda_{S}\beta\dfrac{S^{*}}{N}- \lambda_{I}\left(\beta\dfrac{S^{*}}{N}-\gamma-u^{*}-\mu\right) - \lambda_{R}(\gamma+u^{*}),\\ 
\dfrac{d\lambda_{R}}{dt}&=\mu \lambda_{R},\label{AdjointSystem}
\end{aligned}
\end{equation}
with transversality conditions 
\begin{equation}
\lambda_{j}(t_{f})=0,\quad j=S,I,R.\label{transversality}
\end{equation}
The optimality condition is given as
\begin{equation}
\dfrac{\partial H_{i}}{\partial u}=\dfrac{\partial \sigma_{i}}{\partial u}(u,I^{*})-\lambda_{I}I^{*}+\lambda_{R}I^{*}=0.
\end{equation}
\end{theorem}
For the quadratic cost functions $\sigma_{i}(u,I)$ ($i=1,2$), using Theorem~\ref{theorem:pont}, we can straightforwardly obtain the optimality system which consists of the controlled \emph{SIR} model, the system \eqref{AdjointSystem} for the adjoint variables with transversality conditions \eqref{transversality}, and the characterization of the optimal control (for the details, see the supplementary material \ref{AppendixOptimalitySystem}). The optimality system allows us to obtain a numerical approximation of the optimal control for each of this two cases. For the linear cost functions $\sigma_{i}(u,I)$ ($i=3,4$), the objective functional is a $L_{1}-$type functional less amenable to mathematical analysis; therefore, for simplicity, in this case we only perform a numerical analysis of the optimal controls. The numerical approximations will allow us to perform the cost-effectiveness analysis.

\section{Numerical results for the SIR model}\label{sec:SIRsimulations}
We conduct a numerical study of the optimal control problems using BOCOP 2.2.1 software package \citep{bocop}. In BOCOP, the optimal control problem is approximated by a finite dimensional optimization problem. This is done by a discretization in time applied to the state and control variables, as well as the dynamics equation. These methods are usually less precise than indirect methods based on the Pontryagin's maximum principle, but more robust with respect to the initialization.\par 
\begin{figure}[hbtp]
\vspace*{-70pt}
 \centering
  \subfloat[]{
    \includegraphics[width=0.4\textwidth]{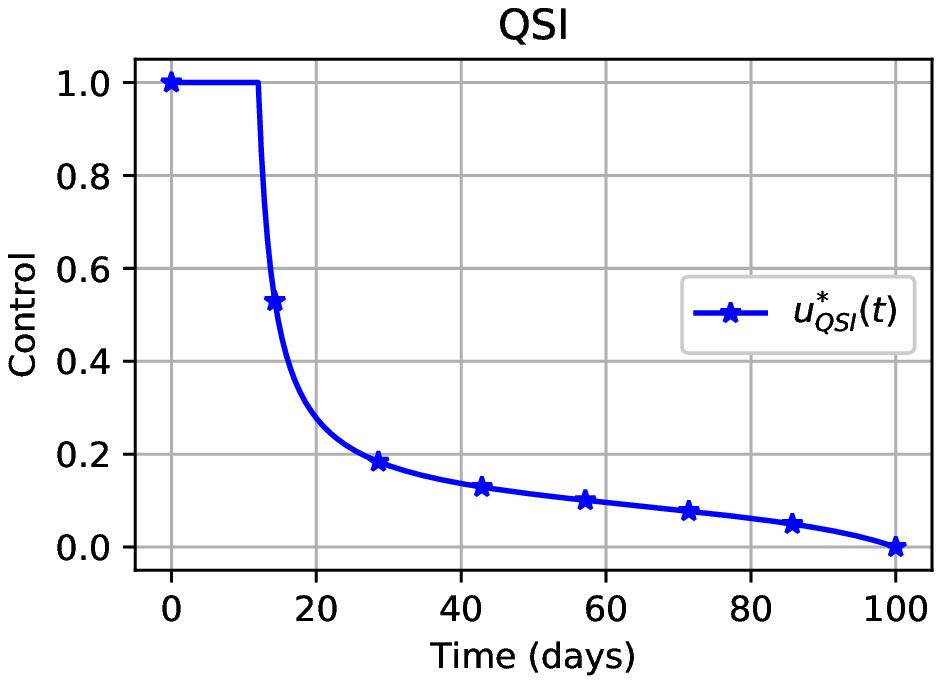}}
  \subfloat[]{
    \includegraphics[width=0.4\textwidth]{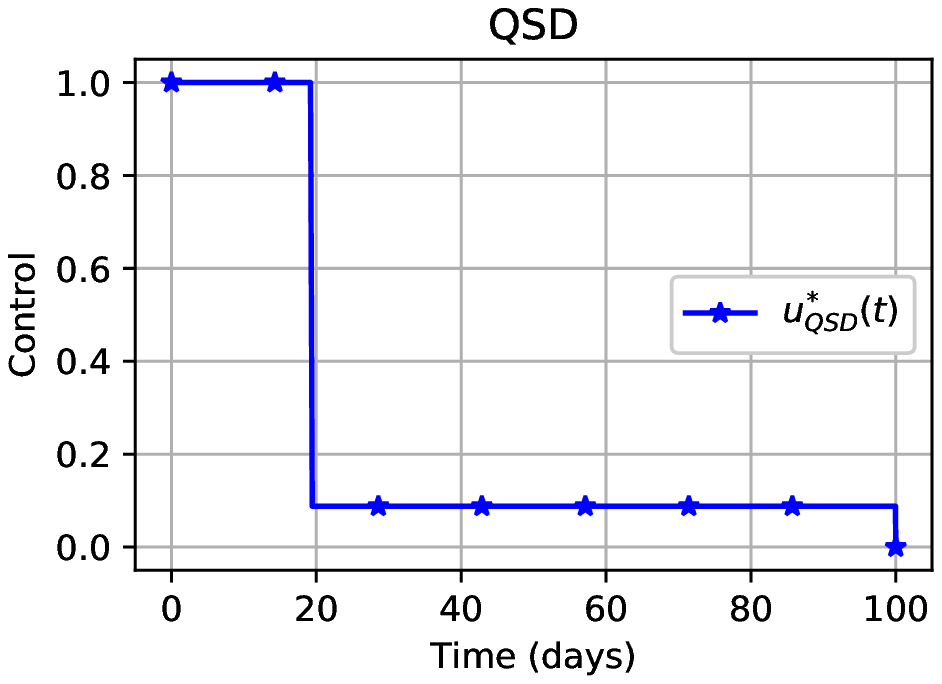}}\\
  \subfloat[]{
    \includegraphics[width=0.4\textwidth]{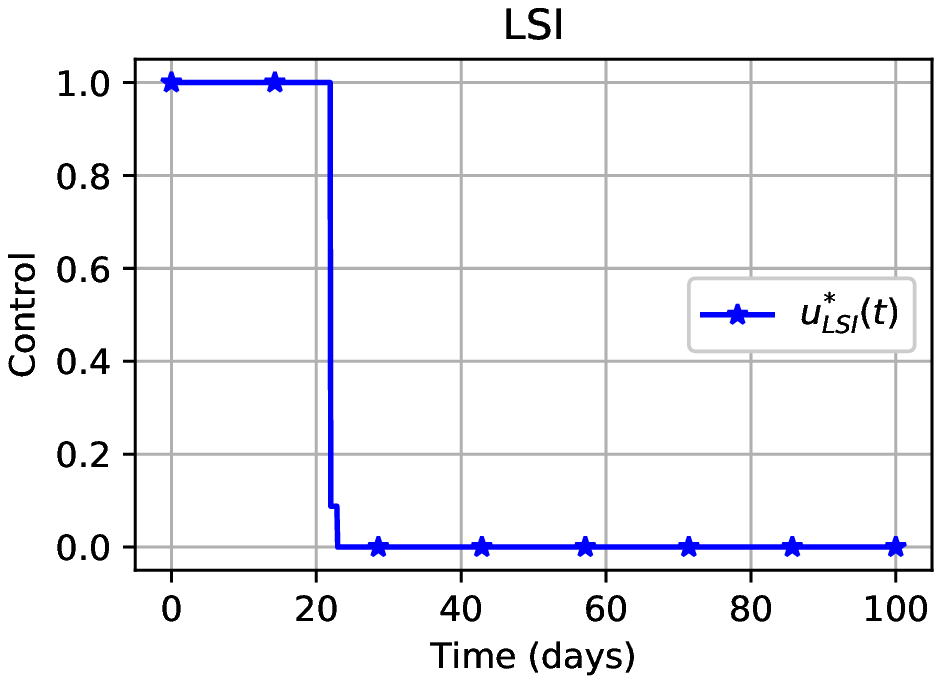}}
  \subfloat[]{
    \includegraphics[width=0.4\textwidth]{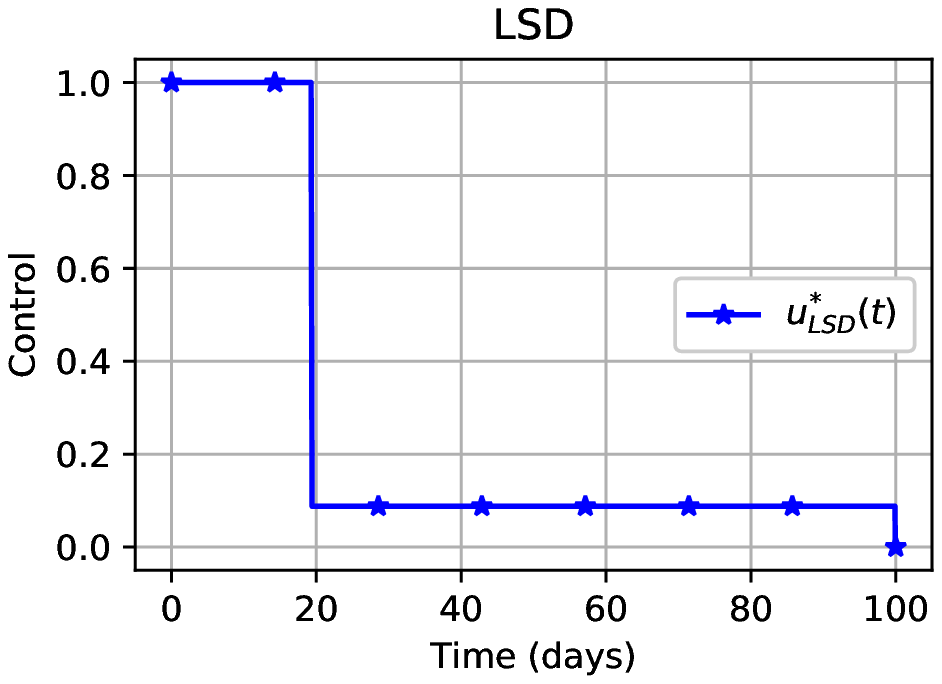}}
 \caption{Associated optimal control of the controlled SIR model \eqref{ControlledSIR} for the QSI  (a),  QSD  (b),  LSI  (c) and  LSD  (d) cases, respectively.}
 \label{fig:controlsSIR}
 
   \subfloat[]{
    \includegraphics[width=0.4\textwidth]{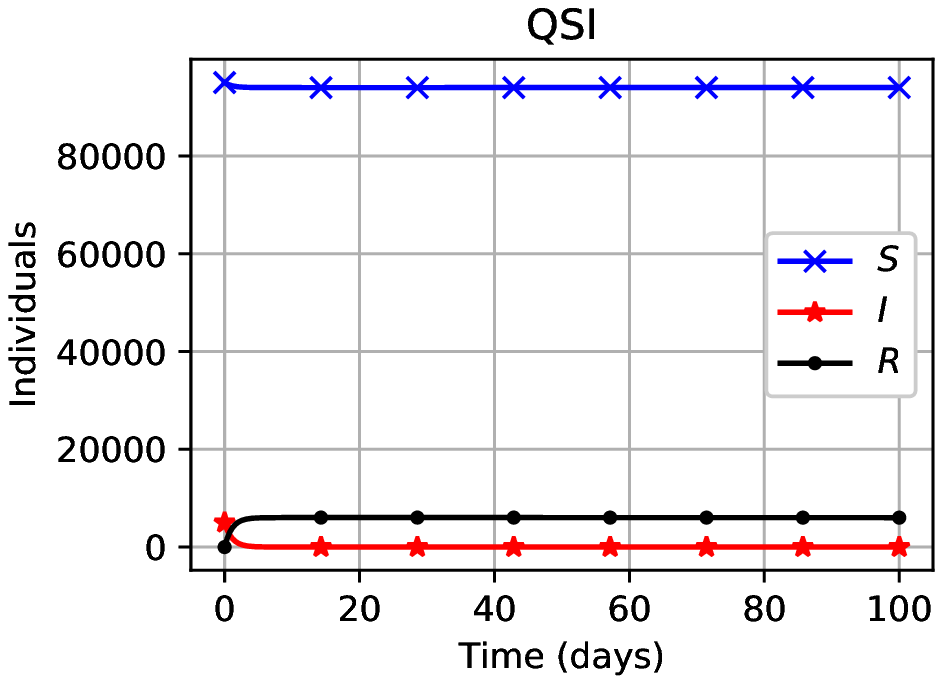}}
  \subfloat[]{
    \includegraphics[width=0.4\textwidth]{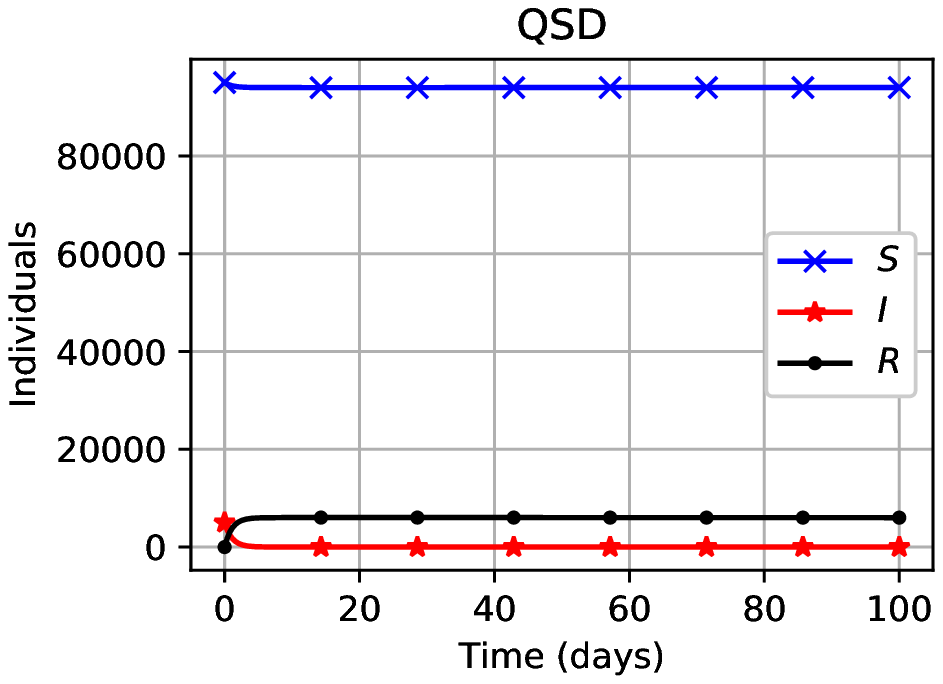}}\\
  \subfloat[]{
    \includegraphics[width=0.4\textwidth]{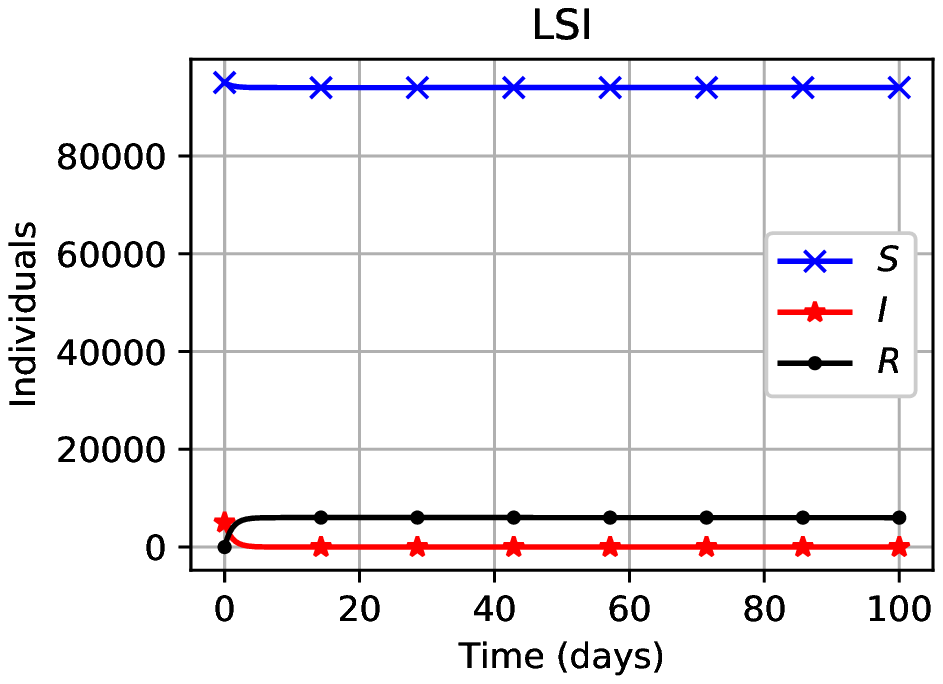}}
  \subfloat[]{
    \includegraphics[width=0.4\textwidth]{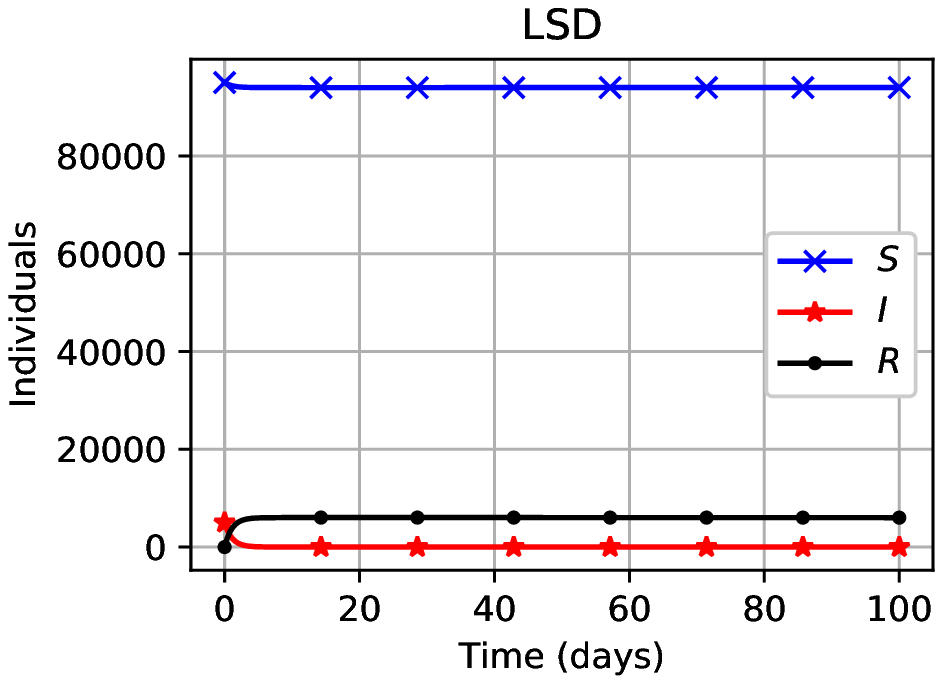}}
 \caption{Associated optimal states of the controlled SIR model \eqref{ControlledSIR} for the QSI  (a),  QSD  (b),  LSI  (c) and LSD  (d) cases, respectively.}
 \label{fig:statesSIR}
\end{figure}\par 
The parameters for the numerical simulations are taken from \citep{rachah2015} (Ebola model): specifically, $\mu=1/(70\times 365)$ (days), $\beta=0.2$ ($1/\text{days}$), and $\gamma=0.1$ ($1/\text{days})$. These parameters were estimated under the assumption that $95\%$ of the population is susceptible and $5\%$ infectious, therefore, we use initial conditions $N_{0}=100000$, $S(0)=95000$, $I(0)=5000$, $R(0)=0$. The final time for our simulations is $t_{f}=100$ days. We consider the optimal control problem \eqref{OCproblem} denoting the optimal control functions $u^{*}_{QSI}$, $u^{*}_{QSD}$, $u^{*}_{LSI}$ and $u^{*}_{LSD}$ for the cases QSI, QSD, LSI and LSD, respectively. The optimal control profiles are shown in Fig. \ref{fig:controlsSIR}, and the associated optimal states for each case are shown in Fig. \ref{fig:statesSIR}.\par 
%
% Explaining the OC results
Fig. \ref{fig:controlsSIR} indicates that the results can be roughly separated into two groups, namely, the results for costs QSI and LSI on one hand, and those for QSD and LSD on another hand. The results for QSD and LSD are very similar. The similarity between QSI and LSI is not so apparent, but the difference is only in the rate of the transmission process between the maximum and minimum values of $u(t)$. (It can be expected that the higher is the value of power $n$ in the term $u^n$ in the objective functional, the slower is the transmission.) The most remarkable is, however, the difference between these two groups and, in particular, the visible difference between the costs QSI and QSD (both these costs are quadratic with respect to $u$), as well as the costs LSI and LSD (both these costs are linear with respect to $u$). This observation leads to a conclusion that, at least for the cases where a cost of intervention policy is low compared with the burden of the disease itself, the optimal controls for the state-independent intervention costs $A_2 u^n $, for all $n \in [1,2]$, should be somewhere between the controls depicted in Fig. \ref{fig:controlsSIR} (a) and Fig. \ref{fig:controlsSIR} (c), and that the impact of the actual value of power $n$ in the cost might be of no principal importance. The same conclusion is even more apparent for the state-dependent intervention costs of the form $A_2 u^n I$, where $n \in [1,2]$. That is, the results in the literature that are obtained for the costs quadratic with respect to control $u(t)$  are valid, at least qualitatively, for the control problems where costs are linear with respect to control and for all intermediate costs $u^n (t)$ where $n \in [1,2]$. This  implies that results for the objective functions that are quadratic with respect to controls, which are comparatively easy to deal with mathematically, are still able to  provide a valuable insight into the qualitative behavior of the optimal controls, and give a correct idea of optimal controls principal properties. We have to stress that this conclusion is valid only for comparatively low-cost intervention policies.\par    
At the same time, the results also indicate that an actual form of the dependency of the intervention costs on the state variables is an issue of principal importance, and that a variation of this can involve significant changes in the outcomes; please compare the controls in Fig. \ref{fig:controlsSIR}(a) and Fig. \ref{fig:controlsSIR}(b).

\section{Cost-effectiveness analysis for the $SIR$ model}\label{sec:CEAsinglecontrol}
Cost-effectiveness analysis is a useful tool that allows to compare the relative health and economic consequences associated with different strategies~\citep{muennig2016}. We compare the cost and health benefits of the control strategies $\pi_{i}$ defined in terms of the optimal control functions $u_{i}^{*}$  (where $i=\lbrace QSI, QSD, LSI, LSD\rbrace$). In order to do this, we need to compute the cost and effectiveness for each of these control strategies using some appropriate health outcome, such as lives saved, the total number of infections averted and years of life gained.\par 
In this study, to measure the effectiveness of an intervention $\pi$, we define the cumulative level of averted infections $E(\pi)$ as the difference between the total numbers of the infected individuals in the absence of control and of those with the control $\pi$:
\begin{equation}
\label{CumuAverted}
E(\pi) = \int_{0}^{t_f} 
\left( I_f(t) - \tilde{I}_f(t) \right) dt.
\end{equation}
Here, $I_f(t)$ and $\tilde{I}_f(t)$ are, respectively, the numbers of infected individuals for the no--control scenario and under intervention $\pi$ at time $t$.\par 
To compute the total cost, we can use any of the objective functionals~\eqref{CostFunctional}. However, one has to take into consideration the rather obvious fact that if we use the objective functional $J_{1}$ with cost function $\sigma_{1}(u,I)$ modeling QSI cost, then the corresponding optimal control $u_{QSI}^{*}(t)$ is the most cost-effective strategy. Therefore, in an attempt to make a fair comparison of the control strategies for different intervention costs, for each of the control strategies we compute four costs using all four objective functionals $J_{i}$, $i=1,2,3,4$. In other words, if $\pi$ is a control intervention defined by the optimal control function $u(t)$, then the total cost of applying $\pi$ is defined as
\begin{equation}
C_{i}(\pi)=J_{i}(u(t))=\int_{0}^{t_{f}}\varphi(\tilde{I}(t))+\sigma_{i}(u(t),\tilde{I}(t))dt, \quad i=1,2,3,4.
\end{equation}\par 
The results of these computations are shown in Table \ref{tab:costos1}.%\par  

\begin{table}[h!]
\centering
\begin{tabular}{lccccc}
\toprule
Strategy &
$C_{1}(\pi_{i})$  &
$C_{2}(\pi_{i})$  &
$C_{3}(\pi_{i})$  &
$C_{4}(\pi_{i})$  &
$E(\pi_{i})$  \\
\toprule
$\pi_{QSI}$ & 548971.74 & 603701.05 & 549060.76 & 603701.38 & 800676.1 \\
$\pi_{QSD}$ & 549009.46 & 603691.32 & 549074.40 & 603691.32 & 800676.2 \\
$\pi_{LSI}$ & 549035.59 & 603696.71 & 549036.47 & 603696.71 & 800675.5 \\
$\pi_{LSD}$ & 549009.69 & 603691.32 & 549074.33 & 603691.32 & 800676.2 \\ %\hline
\bottomrule
\end{tabular}

\caption{Total costs using each of the cost functions $C_{i}$ ($i=1,2,3,4$) and effectiveness (infection averted), for the control strategies.}
\label{tab:costos1}
\end{table}%\par 

Before we proceed to the cost-effectiveness analysis, we have to point out that results in Table~\ref{tab:costos1} confirm the above-made conclusion regarding the qualitative similarity between the optimal controls in strategies  QSI and LSI on one hand, and QSD and LSD on the other hand, and regarding the disparity between these two couples of scenarios. Indeed, for all four strategies, the difference between costs $C_1(\pi_i)$ and $C_3(\pi_i)$ in Table~\ref{tab:costos1} is less than 0.02\%, and it is even smaller for the costs $C_2(\pi_i)$ and $C_4(\pi_i)$. At the same time, for all strategies, the differences between $C_1(\pi_i)$ and $C_2(\pi_i)$ (and, thus, between $C_3(\pi_i)$ and $C_4(\pi_i)$) are about 10\% .    

The cost-effectiveness analysis firstly applies the principle of strong dominance \citep{muennig2016}. According to this principle, if one of the considered interventions is more effective and less expensive, it is said to be dominant. If the intervention is more expensive and less effective than an alternative one, it is said to be dominated. However, for the interventions that are more effective but at the same time are more costly than the alternatives, the principle of strong dominance provides no guidance. Please note that, for the considered model, the effectiveness is practically the same for all four control strategies (see Table \ref{tab:costos1}). This is due to the fact that the associated states are almost the same for each of the four considered control functions. Therefore, the cost-effectiveness analysis is straightforward: the strategy with the lowest cost in the same time is the most cost-effective strategy.\par 
The corresponding ranking of the strategies is given in Table \ref{tab:ranking1}. Please note that each of the $\pi_{i}$ strategies is the most cost-effective strategy with respect to its own associated objective functional. This outcome is hardly unexpected, and implies that the decisive conclusion regarding the cost-effectiveness of the strategies can be made only on the basis of a specific cost function, and that if this cost function is used as the objective functional (as it should be), then the corresponding optimal control and the optimal solution will be the best-performing strategy.

It can be important to identify the strategy exhibiting the second best performance. However, due to the fact that the strategies QSI and LSI on one hand, and QSD and LSD on the other hand, lead to qualitatively similar outcomes, one can expect that the second best performer in the ranking must belong to the same couple as the best one. Table~\ref{tab:ranking2} confirms this expectation for three costs of four considered. (Cost $C_1$ based of the objective QSI is the only cost of the four where strategies QSD and LSD exhibited better cost-efficacy than LSI). 
\begin{table}[h!]
\centering

\begin{tabular}{lcccc}
\toprule
Strategy &
$C_{1}-Rank$  & $C_{2}-Rank$  & $C_{3}-Rank$  & $C_{4}-Rank$  \\
\toprule
$\pi_{QSI}$  & 1          & 4 & \textbf{2} & 4  \\ %\hline
$\pi_{QSD}$  & \textbf{2} & 1 & 4 & \textbf{2}  \\ %\hline
$\pi_{LSI}$  & 4          & 3 & 1 & 3  \\ %\hline
$\pi_{LSD}$  & 3      & \textbf{2} & 3 & 1   \\ %\hline
\bottomrule
\end{tabular}

\caption{Best-performance ranking for the control strategies using each of the four possible cost functions. The strategies with the second best performance are marked in bold.}
\label{tab:ranking1}
\end{table}

\section{Optimal control problem for a $SEIRS$ model with two controls}\label{sec:OCseirs}
The above considered $SIR$ model \eqref{ControlledSIR} with a single control  is one of the simplest epidemic models in the literature. This simplicity makes the problem amenable to mathematical analysis; however, at the same time, due to this simplicity, the influence of different cost functionals on the solution of the optimal control problem was difficult to detect. Therefore, in this section, we extend our study to a more complicated case of the compartmental $SEIRS$ epidemic model, which includes two control functions. We believe that analysis of this model allows to highlight more distinctively the role of a form of the objective functional on the optimal solutions.\par 
The dynamics of $SEIRS$ model is governed by the following system of differential equations:
\begin{equation}
\begin{aligned}
\dfrac{dS}{dt}&=\mu N -\beta \dfrac{SI}{N}-v(t)S +\theta R -\mu S,\\
\dfrac{dE}{dt}&=\beta \dfrac{SI}{N} -\alpha E - \mu E,\\
\dfrac{dI}{dt}&=\alpha E -\gamma I -u(t)I -\mu I,\\
\dfrac{dR}{dt}&=\gamma I + u(t)I + v(t)S - \theta R -\mu R,\label{ControlledSEIRS}
\end{aligned}
\end{equation}
where the total population size $N(t)=S(t)+E(t)+I(t)+R(t)$ is assumed to be constant. %\par 
The  model postulates that the disease has a latent period and, hence, after infection, individuals pass from the susceptible compartment $S$ into the exposed compartment $E$. The individuals in this compartment are infected but not yet infectious. The parameter $\alpha$ is the rate at which the exposed individuals become infectious. The recovered individuals lose the temporary immunity at a rate $\theta$ and return to the susceptible class $S$. The rest of parameters are the same as in the $SIR$ model \eqref{ControlledSIR}. 

Compared with the model \eqref{ControlledSIR}, in this model we incorporate an additional time-dependent control function $v(t)$ that represents the per capita vaccination rate of the susceptible individuals. Accordingly, the vaccinated susceptible individuals enter into the recovered class. We assume that the immunity acquired through the vaccination is also temporary, and that the vaccinated susceptible individuals loss immunity at the same rate as the other recovered individuals. (Such a situation with the loss of immunity occurs if the virus mutates, as in the case of influenza.)
%\par 
%
The control problem involves minimization of the cumulative number of infectious individuals on a given time interval and of the cost of the intervention policy. The objective functional is defined as follows:
\begin{equation}
J_{i}(v(t),u(t))=\int_{0}^{t_{f}}A_{1}I(t)+\sigma_{i}(v(t),u(t),I(t))dt.\label{CostFunctional2}
\end{equation}
Below we consider the following forms for the cost function $\sigma_{i}(v,u,I)$:
\begin{eqnarray}
\sigma_{1}&=&A_{2}(v^{2}+u^{2}), \; \text{quadratic state independent cost (QSI)},\label{c1SEIRS}\\
\sigma_{2}&=&A_{2}(v^{2}S+u^{2}I), \; \text{quadratic state dependent cost (QSD)}\label{c2SEIRS},\\
\sigma_{3}&=&A_{2}(v+u), \; \text{linear state independent cost (LSI)},\label{c3SEIRS}\\
\sigma_{4}&=&A_{2}(vS+uI), \; \text{linear state dependent cost (LSD)}.\label{c4SEIRS}
\end{eqnarray}%\par 
As before, the optimal control problem consists of finding control functions $v^{*}(t)$ and $u^{*}(t)$, such that  
\begin{equation}
J_{i}(v^{*},u^{*})=\min_{(v,u)\in \mathcal{U}}\left\lbrace J_{i}(v(t),u(t))\right\rbrace \;
\text{subject to system}\; \eqref{ControlledSEIRS}\label{OCproblem2}
\end{equation}
where the set of admissible controls is
\begin{equation}
\mathcal{U}=\left\lbrace (v(t),u(t)) \vert (v(t),u(t))\; \text{measureble},\; 0\leq v(t),u(t)\leq 1, \; t\in[0,t_{f}]\right\rbrace .\label{ControlSet2}
\end{equation}\par 
The proof of the existence of an optimal control $(v^{*}(t),u^{*}(t))$ and corresponding optimal solution for problem \eqref{OCproblem2} follows the same argument of the existence proof for the $SIR$ model with a single control, and is therefore omitted.

\subsection{Numerical results for the SEIRS model}
\begin{figure}[hbtp]
\vspace*{-70pt}
 \centering
  \subfloat[]{
    \includegraphics[width=0.4\textwidth]{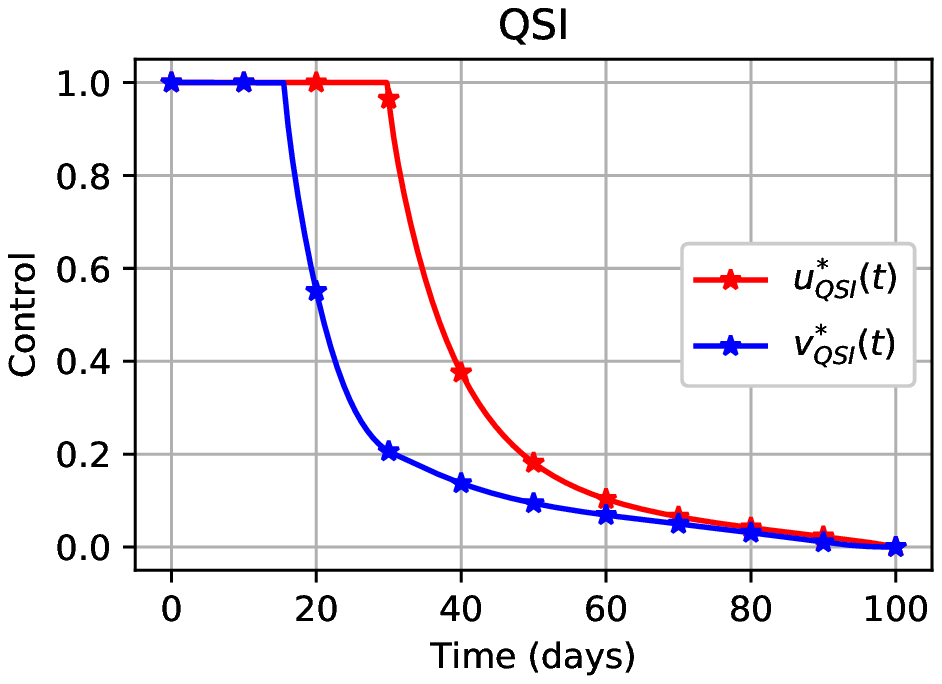}}
  \subfloat[]{
    \includegraphics[width=0.4\textwidth]{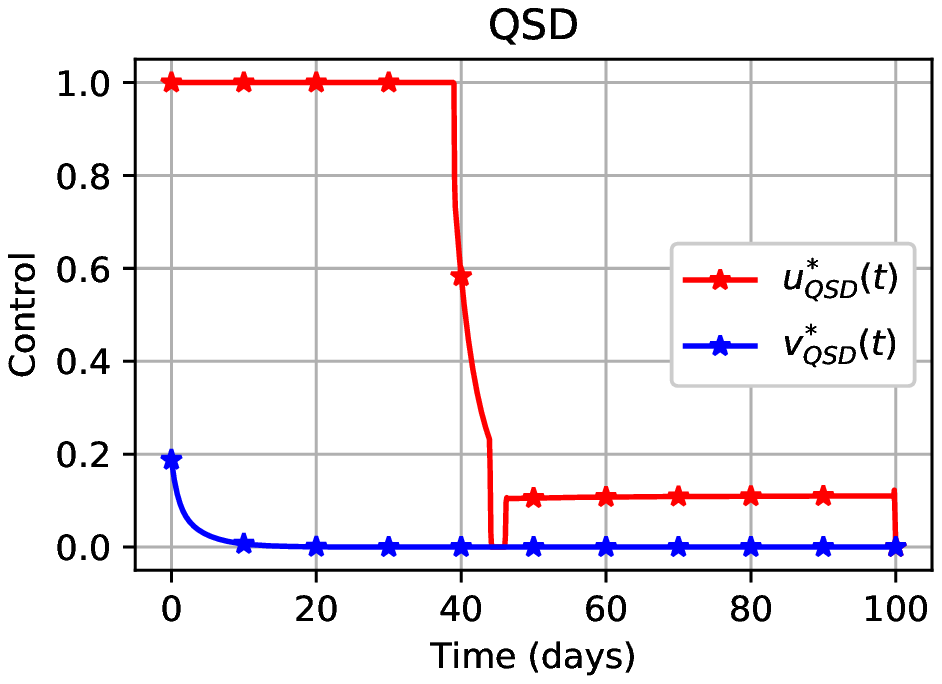}}\\
  \subfloat[]{
    \includegraphics[width=0.4\textwidth]{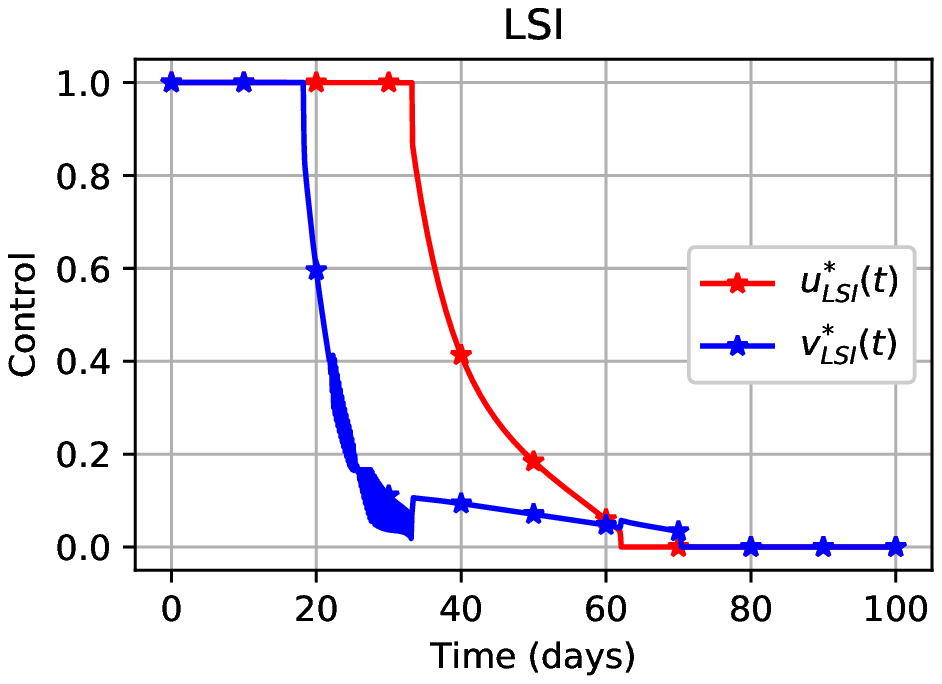}}
  \subfloat[]{
    \includegraphics[width=0.4\textwidth]{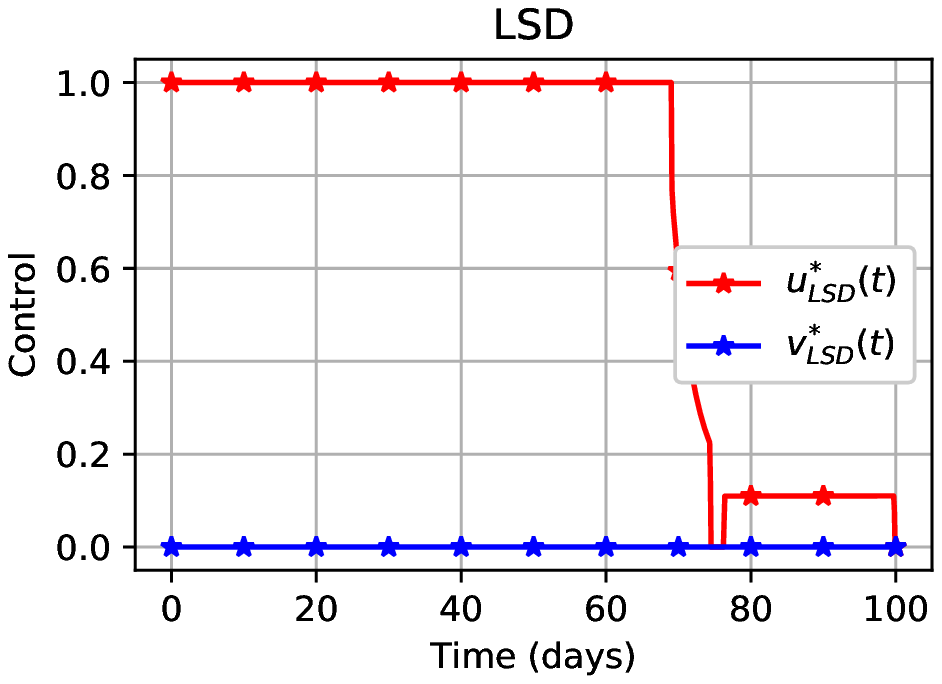}}
 \caption{Associated optimal control pair of the controlled SEIRS model \eqref{ControlledSEIRS} for the QSI  (a),  QSD  (b),  LSI  (c) and  LSD  (d) cases, respectively.}
 \label{fig:controlsSEIRS}
 
   \subfloat[]{
    \includegraphics[width=0.4\textwidth]{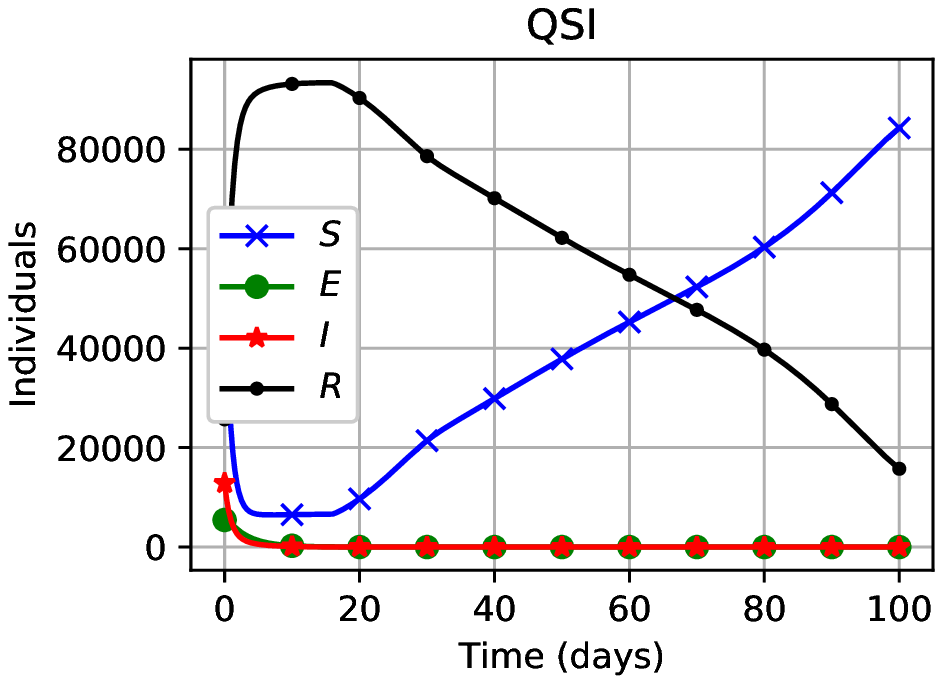}}
  \subfloat[]{
    \includegraphics[width=0.4\textwidth]{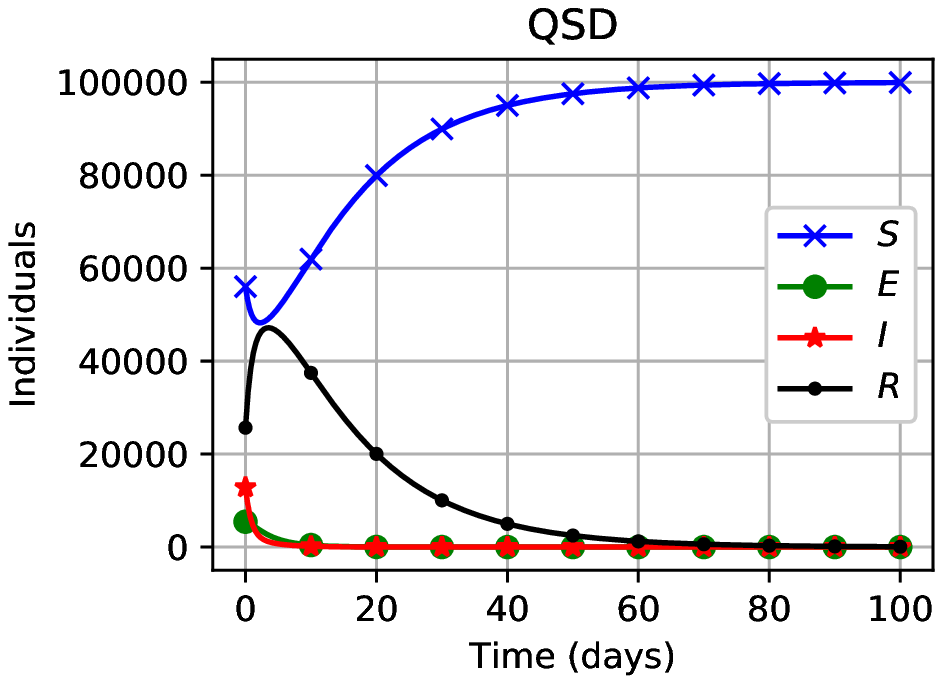}}\\
  \subfloat[]{
    \includegraphics[width=0.4\textwidth]{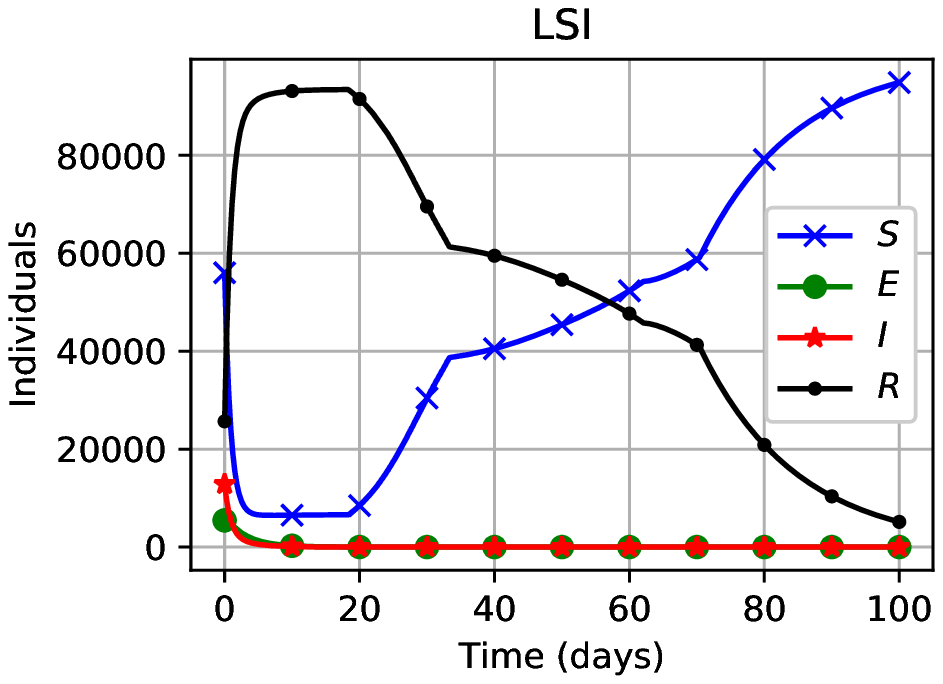}}
  \subfloat[]{
    \includegraphics[width=0.4\textwidth]{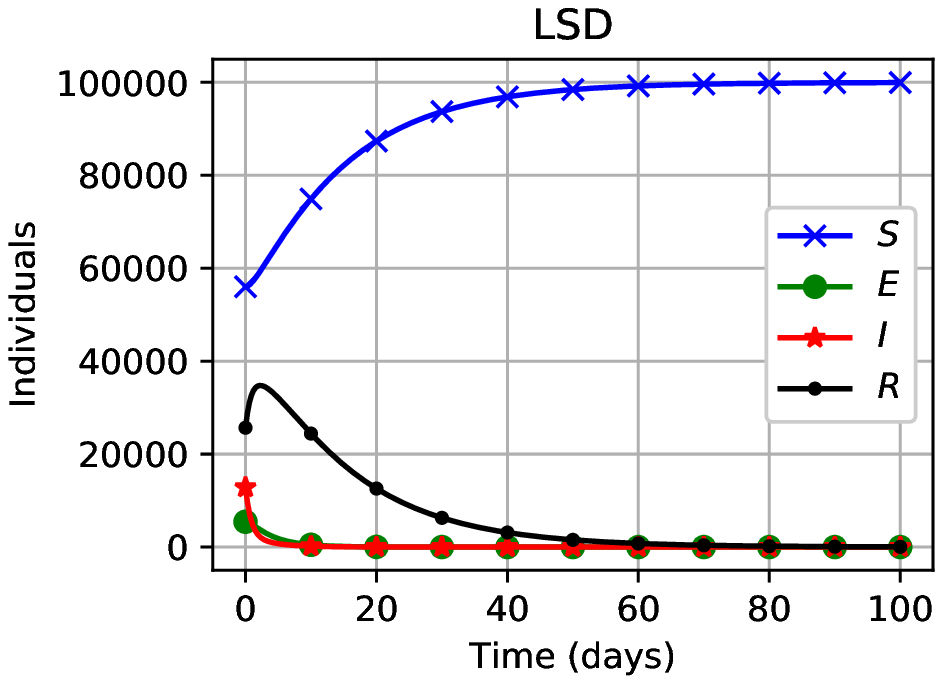}}
 \caption{Associated optimal state of the controlled SEIRS model \eqref{ControlledSEIRS} for the QSI  (a),  QSD  (b),  LSI  (c) and LSD  (d) cases, respectively.}
 \label{fig:statesSEIRS}
\end{figure}\par 
We use the BOCOP 2.2.1 software to obtain numerical approximations of the solutions to the optimal control problem \eqref{OCproblem2} for each of the objective functionals. The parameters for the numerical simulations are taken from \citep{trawicki2017}, specifically, $\mu=0.00003$ ($1/\text{days}$), $\beta=0.25$ ($1/\text{days}$), $\gamma=0.14$ ($1/\text{days}$), $\alpha=0.33$ ($1/\text{days}$), and $\theta=0.07$ ($1/\text{days}$). The final time is $t_{f}=100$ days, and the total population is $N_{0}=100000$. 

Figure \ref{fig:controlsSEIRS} shows the numerical approximations of the control profiles. Please note that the optimal controls for the density independent objective functionals QSI and LSI are similar in qualitative terms. In the same way, the optimal controls for the density dependent cases QSD and LSD also have very similar profiles. The associated optimal states are shown in Figure \ref{fig:statesSEIRS}.\par

\section{Cost-effectiveness analysis for the $SEIRS$ control model}\label{sec:CEAseirs}
To investigate the performance of the control strategies $\pi_{i}$, $i=\lbrace QSD,$ $QSI$, $LSD$, $LSI\rbrace$,  defined in terms of the optimal control pairs $(v_{i}^{*},u_{i}^{*})$, we perform the cost-effectiveness analysis. As discussed above, we measure the effectiveness of a control strategy by the cumulative number of averted infections. To compute the total cost, we use the objective functionals $J_{i}$ \eqref{CostFunctional2}. Hence, if $\pi$ is a control intervention defined by the optimal control pair $(v, u)$, then the total cost of the intervention $\pi$ is defined as
\begin{equation}
C_{i}(\pi)=J_{i}(v,u)=\int_{0}^{t_{f}}A_{1}I+\sigma_{i}(v,u,I)dt, \quad i=1,2,3,4.\label{CostFunctionSEIRS}
\end{equation}
To proceed with the comparison of the interventions, we use the model outcomes to compute the effectiveness and total cost for each of the four considered control strategies. The results are summarized in Table \ref{tab:costos2}. % Table~4 summarizes the best-performance rankings based on the costs in Table~3.
\begin{table}[h!]

\centering
\begin{tabular}{lccccc}
\toprule
Strategy &
$C_{1}(\pi_{i})$  &
$C_{2}(\pi_{i})$  &
$C_{3}(\pi_{i})$  &
$C_{4}(\pi_{i})$  &
$E(\pi_{i})$  \\
\toprule
$\pi_{QSI}$ &
1697872.45 &	3836175.45 &	1698013.74 &	5862075.05 &	1267715.57
 \\
$\pi_{QSD}$ &
1813853.47 &	2012056.49 &	1813920.59 &	2226189.56 &	1266565.05
 \\
$\pi_{LSI}$ &
1697892.72 & 3748946.84 & 1697983.63  & 5155798.23 & 1267715.59 \\
$\pi_{LSD}$ &
1848722.88 &	2032820.69 &	1848757.69 &	2032820.70 &	1266208.63
 \\ %\hline
\bottomrule
\end{tabular}

\caption{Total costs using each of the cost functions $C_{i}$ ($i=1,2,3,4$) and effectiveness (infection averted), for the control strategies.}
\label{tab:costos2}
\end{table}

At first glance, from Table~3 one can note that the effectiveness of the considered strategies does not differ much: the largest difference is of order of 0.1\%. Moreover, $\pi_{QSI}$ and $\pi_{LSI}$ have virtually the same effectiveness. Furthermore, the costs $C_1$  and $C_3$ are virtually the same for all four objective functionals. (The largest difference is smaller than 0.01\%.) At the same time, in contrast to the same data for the $SIR$ model, for the $SEIRS$ model costs $C_2$ and $C_4$ significantly differ, in particularly for strategies $\pi_{QSI}$ and $\pi_{LSI}$. (The difference for $\pi_{QSD}$ is of order of 10\%, whereas for $\pi_{LSD}$ costs $C_2$ and $C_4$ are virtually the same.) This observation implies that for more complicated models, and for the objective functional where the control function appears in the form of a power function $u^n$, the cost of the strategies can depend on the power $n$ more than for a simpler model. However, the outcomes for objectives  QSI and LSI, on one hand, and QSD and LSD, on another hand, still remain qualitatively similar, and, hence, this dependency still remains considerably lower than that on the state variables. (We have to remind that, according to Figs.~\ref{fig:controlsSEIRS} and \ref{fig:statesSEIRS}, the optimal controls themselves exhibit rather weak dependency on the power $n$, while dependency of the state variable is significant.) 

Using the data in Table~\ref{tab:costos2}, we obtain the cost-effectiveness ranking in Table~\ref{tab:ranking2}. 
The most important observation from the table is that it confirms the existence of the stable pairs QSI and LSI, on one hand, and QSD and LSD, on another hand.

As an illustration of how the cost-effectiveness rankings are computed, we proceed to show explicitly the calculations for the $C_4-Rank$ column of Table~\ref{tab:ranking2}. Using the principle of strong dominance, we can rule out intervention $\pi_{QSI}$, which is dominated by intervention $\pi_{LSI}$ yielding the same effectiveness at a lower cost. 

\begin{table}[h!]
\centering
\begin{tabular}{lcccc}
\toprule
Strategy &
$C_{1}-Rank$  & $C_{2}-Rank$  & $C_{3}-Rank$  & $C_{4}-Rank$  \\
\toprule
$\pi_{QSI}$  
& 1 & 4 & \textbf{2} & 4  \\ %\hline
$\pi_{QSD}$  
& 3 & 1 & 3 &  \textbf{2}   \\ %\hline
$\pi_{LSI}$  
& \textbf{2} & 3 & 1 & 3  \\ %\hline
$\pi_{LSD}$  
& 4 & \textbf{2} & 4 & 1   \\ %\hline
\bottomrule
\end{tabular}
\caption{Best-performance ranking for the control strategies using each of the four possible cost functions. The strategies with the second best performance are marked in bold.}
\label{tab:ranking2}
\end{table}

% ICER-ACER info should be modified
However, we have to observe that, for the \emph{SEIRS} model, the principle of strong dominance is not sufficient to determine the most cost-effective intervention, and, therefore, more sophisticated techniques should be applied. When the more effective strategy is also more costly, the decision maker must decide if the greater effectiveness justifies the cost of achieving it. Usually, this can be done by calculating the incremental cost-effectiveness ratio (ICER) of two (or more) strategies, which is defined as the difference in the costs divided by the difference in the effectiveness. In mathematical terms, for two interventions $\pi_1$ and $\pi_2$, 
\begin{equation}
\label{ICER-two}
ICER(\pi_1, \pi_2) = \frac{C(\pi_2) - C(\pi_1)}{E(\pi_2) - E(\pi_1)},
\end{equation}
provided that $E(\pi_1) \neq E(\pi_2)$. The ICER of each intervention is found by comparing it to the next most effective intervention. The ICER are summarized in Table \ref{tab:CEA_explanation}).\par

\begin{table}[h!]
\centering

\begin{tabular}{lccc}
\toprule
Strategy   &  $C_{4}(\pi_{i})$  & $E(\pi_{i})$  & ICER        \\
\toprule
$\pi_{LSI}$  &  5155798.24 & 1267715.59   & $ICER(\pi_{LSI},\pi_{QSD})=2547.48$  \\
$\pi_{QSD}$  &  2226189.56 &	1266565.05 & $ICER(\pi_{QSD},\pi_{LSD})=542.53$   \\ %\hline
$\pi_{LSD}$  &  2032820.70 &	1266208.63 & \\ %\hline
\bottomrule
\end{tabular}

\caption{Interventions listed in the order of effectiveness together with their ICER computed by comparing with the next most effective intervention. ($\pi_{QSI}$ is omitted as dominated by $\pi_{LSI}$.)}
\label{tab:CEA_explanation}
\end{table}\par

Alongside to the ICER, we also introduce the average cost-effectiveness ratio (ACER) as a willingness-to-pay threshold. The ACER is determined with respect to the ``no effect''/``no cost'' alternative, and is mathematically given as \citep{phillipsCEA}:
\begin{equation}
\label{ACER}
ACER(s) = \dfrac{C(S)}{E(S)}.
\end{equation}%\par 
As a willingness-to-pay threshold, we use ACER of the less costly intervention, which in this case is $\pi_{LSD}$ (see Table~5):
\begin{equation}
ACER(\pi_{LSD}) = \dfrac{2032820.70}{1266208.63} = 1.6054.
\end{equation}
Since $ICER(\pi_{LSI}, \pi_{QSD}) >\; ACER(\pi_{LSD})$ and $ICER(\pi_{QSD}, \pi_{LSD}) >\; ACER(\pi_{LSD})$, the strategy $\pi_{LSD}$ is the most cost-effective strategy. Similarly, we compute the ACER for $\pi_{QSD}$,
\begin{equation}
ACER(\pi_{QSD}) = \dfrac{2226189.56}{1266565.05} = 1.7576.
\end{equation}
Since $ICER(\pi_{LSI}, \pi_{QSD}) >\; ACER(\pi_{LSD})$, the strategy $\pi_{QSD}$ is the second most cost-effective strategy, and thus $\pi_{QSD}$ has Rank \#2, $\pi_{LSI}$ has Rank \#3, and $\pi_{QSI}$ has Rank \#4.\par 

\section{Discussion}\label{sec:discussion}
There are hardly any doubts that the optimal control theory applied to the problems originated in medicine and biology can potentially bring significant benefits. At the same time, so far, such applications are rather limited and generally undervalued. The authors believe that one of the reasons for this is a controversy with objectives of a control strategy that are formulated in the form of objective functionals. The problem, as the authors see it, is that the vast majority of publications in the literature consider objective functionals that include a sum of weighted squares of the controls. While the functionals of such a form lead to mathematically convenient problems, their biomedical interpretation is often uncertain and even dubious. Besides, the objective functionals where the controls are included as a function $u^n$ with $n \neq 2$ (and, in particular, with $n=1$, that appears to be the most natural case for many biomedical applications) imply the necessity of dealing with a considerably more complicated mathematical problem and, therefore, are rather rare in the literature. This controversy leads to a question regarding the extent to which results obtained from a specific objective functional are valid and can be trusted. Unfortunately, despite its obvious practical relevance, so far this issue has attracted rather small attention: so far the authors were able to find only two publications~\citep{ledzewicz2004comparison,ledzewicz2020comparison} that address this issue. Both publications are of the same authors and are dealing with within-a-patient disease dynamics.

% Summarizing the paper
%
To further address this issue, in particularly concerning infection control at a level of the entire population, in this paper we explored the impact of a form of the objective functionals with respect to controls on the structure of the optimal controls and the associated optimal solutions for two reasonably simple classical   epidemiological models.   Specifically, we considered a \emph{SIR} and a \emph{SEIRS} models of the spread of an infectious disease in a population. For these models, we considered the objective functionals where the cost of control action were assumed to be a linear and a quadratic functions ($L_2$- and $L_1$-type functionals). Moreover, for both these types, we considered the costs that depend of the density of individuals and are density-independent. To quantitatively compare the impacts of specific forms of the objective functionals on the optimal controls, in this paper we also perform cost-effectiveness analysis. 

% Our results
% 

Our numerical simulations show that for both, the $SIR$ model \eqref{ControlledSIR} and the $SEIRS$ model \eqref{CostFunctionSEIRS}, the optimal controls for the both considered density-independent cases, where the controls are included as a quadratic (the QSI scenario) and a linear (the LSI scenario) functions, are qualitatively similar. Likewise, the optimal controls and optimal solutions for both density-dependent cases, where the controls are included as quadratic (the QSD scenario) and linear (the LSD scenario) functions, have very similar profiles as well. At the same time, the optimal controls and the optimal solutions for the density-independent and the density-dependent cases significantly differ. This observation allows us to make a very important conclusion that, for the objective functionals that include the controls in the form of a power function $u^n$, and at least for the cases where the cost of intervention is notably lower than the potential losses inflicted by the infection (10\% in our cases), the optimal controls and optimal solutions are robust with respect to the power $n$.

The second conclusion that follows from our results is that the dependency of the cost of intervention on the state variables (that is, on the densities of the considered classes) has a  remarkably large impact on the optimal control and the corresponding optimal solution. Hence,  for application of the optimal control theory to real-life practice, a specific form of the dependency of the costs on the state variables is significantly more important than the dependency of the cost on the controls themselves. These conclusions indicate that, for a specific predetermined form of the dependency of the cost of intervention on the state densities, the most common in the literature results obtained for the intervention costs that depend on the squares of the controls are still able to provide valuable insights and can serve as a guideline for the cases where the real-life dependency of the cost on the controls is uncertain or is assumed to be a non-quadratic. We have to stress again that this conclusion is valid for the cases where the cost of intervention is notably lower than the losses inflicted by the infection. At the same time, we also have to stress that, while, for a specific predetermined form of the dependency of the cost on the densities, the optimal controls and optimal solutions are reasonably robust with respect to a variation of the power $n$, the actual cost can be notably different for different values of the power $n$.

\bibliographystyle{apalike}
\bibliography{references}

%% To change the style, put a % in front of the second line of the current style and
%% remove the % from the second line of the style you would like to use.
%%%%%%%%%%%%%%%%%%%%%%%

%% Numbered

%% Numbered without titles
%\bibliographystyle{model1a-num-names}

%% Harvard
%\bibliographystyle{model2-names.bst}\biboptions{authoryear}

%% Vancouver numbered
%\usepackage{numcompress}\bibliographystyle{model3-num-names}

%% Vancouver name/year
%\usepackage{numcompress}\bibliographystyle{model4-names}\biboptions{authoryear}

%% APA style
%\bibliographystyle{model5-names}\biboptions{authoryear}

%% AMA style
%\usepackage{numcompress}\bibliographystyle{model6-num-names}

%%%%%%%%%% Merge with supplemental materials %%%%%%%%%%
%\widetext
%\clearpage

\appendixpageoff
\appendixtitleoff
\begin{appendices}
\setcounter{section}{0}
  
\begin{center}
\textbf{\large Supplemental Materials: }

{\large
Impact of a cost functional on the optimal control and the cost-effectiveness: control of a spreading infection as a case study}

\medskip

Fernando Salda\~na$^1$, Ariel Camacho$^2$, Andrei Korobeinikov$^3$

\small{
$^1$ \textit{Instituto de Matem\'aticas, Campus Juriquilla, 76230, Universidad Nacional Aut\'onoma de M\'exico, Qu\'eretaro, Mexico}

$^2$ \textit{Facultad de Ciencias, Universidad Aut\'onoma de Baja California, 22860 Baja California, Mexico}

$^3$ \textit{School of Mathematics and Information Science, Shaanxi Normal University, Xi\'an, China}}

\end{center}
%%%%%%%%%% Merge with supplemental materials %%%%%%%%%%
%%%%%%%%%% Prefix a "S" to all equations, figures, tables and reset the counter %%%%%%%%%%
\setcounter{equation}{0}
\setcounter{figure}{0}
\setcounter{table}{0}
\setcounter{page}{1}
\makeatletter
\renewcommand{\theequation}{S\arabic{equation}}
\renewcommand{\thefigure}{S\arabic{figure}}
\renewcommand{\bibnumfmt}[1]{[S#1]}
\renewcommand{\citenumfont}[1]{S#1}
%%%%%%%%%% Prefix a "S" to all equations, figures, tables and reset the counter %%%%%%%%%%

%
\section{Existence of the optimal control}\label{AppendixExistence}
Let $x=(S,I,R)^{T}$ and denote the right-hand side of system \eqref{ControlledSIR} as the vector function $F(t,x,u)$. The optimal control problem is
\begin{equation}
\min_{u\in \mathcal{U}}J_{i}(u)=\min_{u\in \mathcal{U}}\int_{0}^{t_{f}}L_{i}(t,x,u)dt \quad
\text{subject to system}\; \eqref{ControlledSIR}, \label{OCproblemAppendix}
\end{equation}
where
\begin{equation}
L_{i}(t,x,u)=A_{1}I+\sigma_{i}(u,I)
\end{equation}
and cost functions $\sigma_{i}(u,I)$, $i=1,2,3,4$, are  given by  equations \eqref{c1}--\eqref{c4}. 
To prove the existence of an optimal control, we employ Theorem $4.1$ from Fleming and Rishel \citep[Chapter III]{Fleming1975}. This existence theorem states that the following conditions are sufficient to guarantee the existence of an optimal control for \eqref{OCproblemAppendix}:
\begin{itemize}
\item[(H1)] $F$ is continuous, and there exist positive constants $K_{1}$ and $K_{2}$ such that
\begin{itemize}
\item[(a)] $\vert F(t,x,u)\vert \leq K_{1}(1+\vert x\vert + \vert u\vert)$
\item[(b)] $\vert F(t,\tilde{x},u)-F(t,x,u)\vert \leq K_{2}\vert \tilde{x}-x\vert (1+\vert u\vert)$
\end{itemize}
hold for all $t\in[0,t_{f}]$. Moreover, $F$ can be written as $F(t,x,u)=\phi(t,x)+\psi(t,x)u$.
\item[(H2)] The set of admissible controls $\mathcal{U}$ is closed and convex. Moreover, there is at least one feasible pair $(x(t), u(t))$ satisfying both \eqref{ControlledSIR} and $u(t)\in\mathcal{U}$.
\item[(H3)] $L_{i}(t,x,\cdot)$, $i=1,2,3,4$ is convex on $\mathcal{U}$, and $L_{i}(t,x,u)\geq g(u)$, where $g$ is continuous and $\vert u\vert^{-1}g(u) \rightarrow +\infty$ as $\vert u\vert \rightarrow \infty$, $u\in \mathcal{U}$. 
\end{itemize}
\begin{theorem}\label{TheoremExistence}
Consider the optimal control problem \eqref{OCproblemAppendix} with control model \eqref{ControlledSIR} and cost functional $J_{i}(u)=\int_{0}^{t_{f}}L_{i}(t,x,u)dt$, $i=1,2,3,4$. Then there exist an admissible control function $u^{*}(t)\in \mathcal{U}(t_{f})$ such that $\min J_{i}(u)_{u\in \mathcal{U}}=J_{i}(u^{*})$.
\end{theorem}
\begin{proof}
Since we are considering a constant population, the solution of model \eqref{ControlledSIR} are bounded. In addition,  $F \in C^{1}$, and, therefore, (a) and (b) in (H1) are ensured by suitable bounds on the partial derivatives of $F$ and on $F(t, 0, 0)$. Moreover, the state equations are linear with respect to the controls $u$, and, thus, $F(t,x,u)=\phi(t,x)+\psi(t,x)u$. The existence of a feasible pair is guaranteed by the Caratheodory theorem \citep[pp.~182]{lukes1982} for Cauchy problems. Moreover, for bounded controls on a finite time interval, $\mathcal{U}$ is clearly closed and convex, and, hence, (H2) is satisfied as well.
The integrand $L_{i}(t,x,\cdot)$ of the cost functional is non-negative for $i=1,2,3,4$, quadratic for $i=1,2$ and linear for $i=3,4$ with respect to the controls. Therefore, $L(t,x,\cdot)$ is convex on $\mathcal{U}$. Furthermore, the set of admissible controls $\mathcal{U}$ is bounded (this also implies that the term $\vert u\vert$ in (H1) can be omitted), and, hence, (H3) hold vacuously, thus completing the proof.
\end{proof}

\section{Characterization of the optimal control}\label{AppendixOptimalitySystem}
Here, we characterize the optimal control by direct use of the maximum principle for the four cost functions \eqref{c1}--\eqref{c4}. Our aim is to obtain the so-called optimality system, which consist of the controlled \emph{SIR} model, the system \eqref{AdjointSystem} for the adjoint variables with transversality conditions \eqref{transversality}, and the characterization of the optimal control. The optimality system allow us to obtain a numerical approximation of the optimal control $u^{*}(t)$. 
\subsection{Quadratic state-independent cost}\label{QSIappendix}
Let us consider the quadratic state-independent cost function
\begin{equation*}
\sigma_{1}(u,I)=A_{2}u^{2}.
\end{equation*} 
The Hamiltonian is defined by the following equation:
\begin{equation}
\begin{aligned}
H_{1} &= A_{1}I + A_{2}u^{2}
  + \lambda_{S}\left( \mu N -\beta \dfrac{SI}{N}-\mu S\right)\\
  &+ \lambda_{I}\left(\beta \dfrac{SI}{N}-\gamma I -u(t)I -\mu I\right) 
  + \lambda_{R}\left(\gamma I + u(t) I -\mu R\right) ,
\end{aligned}
\end{equation}
The system for the adjoint variables is
\begin{equation}
\begin{aligned}
\dfrac{d\lambda_{S}}{dt}&=\lambda_{S}\left( \beta\dfrac{I^{*}}{N}+\mu\right) -\lambda_{I}\beta\dfrac{I^{*}}{N},\\
\dfrac{d\lambda_{I}}{dt}&=-A_{1}+ \lambda_{S}\beta\dfrac{S^{*}}{N}- \lambda_{I}\left(\beta\dfrac{S^{*}}{N}-\gamma-u^{*}-\mu\right) - \lambda_{R}(\gamma+u^{*}),\\ 
\dfrac{d\lambda_{R}}{dt}&=\mu \lambda_{R},
\end{aligned}
\end{equation}
with transversality conditions $\lambda_{S}(t_{f})=\lambda_{I}(t_{f})=\lambda_{R}(t_{f})=0$.\par 
The optimality condition for this case is
\begin{equation}
\dfrac{\partial H_{1}}{\partial u}=2A_{2}u-\lambda_{I}I^{*}+\lambda_{R}I^{*}=0.
\end{equation}\par 
From the maximum principle we obtain the following characterization for the optimal control:
\begin{equation}
u^{*}(t)=\min \left\lbrace 1, \max \left\lbrace 0, \dfrac{1}{2A_{2}} (\lambda_{I}-\lambda_{R})I^{*}\right\rbrace \right\rbrace.
\end{equation}
\subsection{Quadratic state-dependent cost}\label{QSDappendix}
Here, we consider the quadratic state-dependent cost function
\begin{equation*}
\sigma_{2}(u,I)=A_{2}u^{2}I.
\end{equation*} 
The Hamiltonian is given by the following equation:
\begin{equation}
\begin{aligned}
H_{2} &= A_{1}I + A_{2}u^{2}I
  + \lambda_{S}\left( \mu N -\beta \dfrac{SI}{N}-\mu S\right)\\
  &+ \lambda_{I}\left(\beta \dfrac{SI}{N}-\gamma I -u(t)I -\mu I\right) 
  + \lambda_{R}\left(\gamma I + u(t) I -\mu R\right) ,
\end{aligned}
\end{equation}
The system for the adjoint variables is
\begin{equation}
\begin{aligned}
\dfrac{d\lambda_{S}}{dt}&=\lambda_{S}\left( \beta\dfrac{I^{*}}{N}+\mu\right) -\lambda_{I}\beta\dfrac{I^{*}}{N},\\
\dfrac{d\lambda_{I}}{dt}&=-A_{1} -A_{2}u^{2} +\lambda_{S}\beta\dfrac{S^{*}}{N}- \lambda_{I}\left(\beta\dfrac{S^{*}}{N}-\gamma-u^{*}-\mu\right) - \lambda_{R}(\gamma+u^{*}),\\ 
\dfrac{d\lambda_{R}}{dt}&=\mu \lambda_{R},
\end{aligned}
\end{equation}
with transversality conditions $\lambda_{S}(t_{f})=\lambda_{I}(t_{f})=\lambda_{R}(t_{f})=0$.\par 
The optimality condition for this case is
\begin{equation}
\dfrac{\partial H_{2}}{\partial u}=2A_{2}uI^{*}-\lambda_{I}I^{*}+\lambda_{R}I^{*}=0.
\end{equation}\par 
From the maximum principle we obtain the following characterization for the optimal control:
\begin{equation}
u^{*}(t)=\min \left\lbrace 1, \max \left\lbrace 0, \dfrac{1}{2A_{2}} (\lambda_{I}-\lambda_{R}) \right\rbrace \right\rbrace.
\end{equation}
%

%%
%% End of file `elsarticle-template-num.tex'.

\end{appendices}

\end{document}